\documentclass[a4paper,11pt]{article}
\usepackage{amsmath,amsfonts,amsthm, amssymb,xspace}
\usepackage{geometry, color}
 \usepackage{epigraph}

\usepackage{etoolbox}

\makeatletter
\patchcmd{\epigraph}{\@epitext{#1}}{\itshape\@epitext{#1}}{}{}
\makeatother
\newcommand{\Q}{{\mathbb Q}}
\newcommand{\Z}{{\mathbb Z}}

\newcommand{\Sym}{{\mathcal S}}
\newcommand{\F}{{\rm F}}

\newcommand{\FIN}{{\rm FIN}}
\newcommand{\DA}{{\rm DA}}
\newtheorem{theorem}{Theorem}[section] 
\newtheorem{lemma}[theorem]{Lemma}     
\newtheorem{proposition}[theorem]{Proposition}
\newtheorem{corollary}[theorem]{Corollary}
\newtheorem{definition}[theorem]{Definition}

\newenvironment{proofof}[1]{\normalsize {\it Proof of #1}.}{{\hfill $\Box$}}
\newenvironment{mylist}{\begin{list}{}{
\setlength{\parskip}{0mm}
\setlength{\topsep}{2mm}
\setlength{\parsep}{0mm}
\setlength{\itemsep}{0.5mm}
\setlength{\labelwidth}{7mm}
\setlength{\labelsep}{3mm}
\setlength{\itemindent}{0mm}
\setlength{\leftmargin}{12mm}
\setlength{\listparindent}{6mm}
}}{\end{list}}
\def\margin_comment#1{\marginpar{\sffamily {\small #1\par}\normalfont}}

\begin{document}
\title{Equations in groups that are virtually direct products}

\author{Laura Ciobanu, Derek Holt and Sarah Rees}
\date{}
\maketitle
\begin{abstract}
In this note, we show that the satisfiability of equations and inequations
with recognisable constraints is decidable in groups that are virtually direct
products of finitely many hyperbolic groups. 

\end{abstract}
\epigraph{Dedicated to Charles Sims, who introduced the first author to
equations in groups.}

\section{Introduction}

For any group $G$ and set of variables $\mathcal{Y}$, an \emph{equation} with
coefficients $g_1,\dots, g_{m+1}$ from the group $G$ is a formal expression
$g_1Y_1^{\epsilon_1}g_2 Y_2^{\epsilon_2}\dots Y_m^{\epsilon_m}g_{m+1}=1_G,$
where $\epsilon_j = \pm 1$ for all $1 \leq j \leq m$,
and $Y_j \in \mathcal{Y}.$ 
Such an equation is called \textit{satisfiable} if there exist values for
the $Y_j$'s in $G$ with which the above identity in $G$ is satisfied;
each such set of values for the $Y_j$ is a \textit{solution}.  
Analogously, an \emph{inequation} has the form
$g_1Y_1^{\epsilon_1}g_2 Y_2^{\epsilon_2}\dots Y_m^{\epsilon_m}g_{m+1}\neq1_G$.
A finite set of equations and inequations with coefficients in $G$
is a \textit{system of equations and inequations} over $G$, and is satisfiable if there are assignments to the $Y_j$
such that all of the equations and inequations in the system are satisfied.
For a group $G$, we say that systems of equations and inequations over $G$
are \emph{decidable} over $G$ if there is an algorithm to determine whether any
given such system is satisfiable.  This question is widely known as the
{\em Diophantine Problem} for $G$.

This article investigates equations in groups that are virtually direct
products, and hence addresses the
natural question of whether the decidability of equations
extends from a group $G$ to a group that contains $G$ as a subgroup of finite
index. While the decidability of equations in free groups was established in the 1980s by Makanin \cite{M}, it was only shown in 2010 that the same holds for virtually free groups:  Dahmani and Guirardel \cite{DahmaniGuirardel}
reduced the Diophantine Problem in virtually free groups to the same question relating to
systems of twisted equations and inequations in free groups, and using difficult
topological arguments they proved such systems to be decidable; a different approach to the Diophantine Problem in virtually free groups can be found in \cite{LohreySenizergues}.

Here, in Theorem~\ref{thm:decide_abhyps}, we settle the Diophantine Problem in
any group that is virtually a direct product of a finitely generated abelian
group and non-elementary hyperbolic groups
(equivalently, virtually a direct product of hyperbolic groups).
We do this not by extending the result 
from the finite index direct product, but by embedding the whole group into 
a direct product of permutational wreath products where the above questions
can be answered positively (Lemma \ref{lem:embed_wr_prod}).

In fact, we prove Theorem~\ref{thm:decide_abhyps} for an extended form of
the Diophantine Problem, which asks if it can be decided whether a given system
of equations and inequations has solutions in which some of the variables are
constrained to lie in specified recognisable subsets of the group. (We define
recognisable subsets in the next section.)

Our results show that the Diophantine Problem with recognisable constraints 
can be answered positively for, amongst others, dihedral (i.e. $2$-generator)
Artin groups, and groups that are virtually certain types of right-angled
Artin groups.  We note that, since any dihedral Artin group can alternatively
be decomposed as a central extension of $\Z$ by a virtually free group,
decidability of its systems of equations (but not the more general problem
with recognisable constraints) could also be derived from \cite{Liang}.

\section*{Acknowledgements}

The authors acknowledge support from an LMS `Celebrating New Appointments' grantfor
the meeting \emph{Combinatorics and Computation in Groups}, where discussions
on this paper started, and thank the ICMS for hosting the meeting. They
would also like to thank Volker Diekert and Rick Thomas for clarifications and
references on constraints.

\section{Background and notation}
\label{sec:background}

Let $G$ be a group with finite inverse closed generating set $S$,
and let $\pi: S^* \rightarrow G$ be the natural homomorphism to $G$ from the free
monoid $S^*$ generated by $S$.
When $w$ is a word over $S$, we write $|w|$ to denote the length of $w$ as a 
word and $|\pi(w)|_G$ to denote the length of the shortest word over $S$ that represents $\pi(w)$.

\medskip

\begin{definition}\label{def:rat_con}
\ \begin{mylist}
\item[(1)] 
A subset $L$ of $G$ is said to be {\em recognisable} if the full preimage
$\pi^{-1}(L)$ is a regular subset of $S^*$.
\item[(2)] A subset $L$ of $G$ is said to be  {\em rational} if $L$ is the
image $\pi(L')$  of a regular subset $L'$ of $S^*$.
\item[(3)] A regular subset $L'$ of $S^*$ is
\emph{quasi-isometrically embedded} (q.i. embedded) in $G$ if there exist
$\lambda \geq 1$ and $\mu \geq 0$ such that, for any $w \in L'$,
$|\pi(w)|_G \geq \frac{1}{\lambda}|w|-\mu.$
\item[(4)] A rational subset $L$ of $G$ is \emph{quasi-isometrically
embeddable} (q.i. embeddable) in $G$ if there exists a quasi-isometrically embedded
regular subset  $L'$ of $S^*$ such that $\pi(L')=L.$
\end{mylist}
\end{definition}
It follows immediately from the definition that recognisable subsets of $G$ are
rational.

By \cite[Proposition 6.3]{Thomas} a subset of
$G$ is recognisable if and only if it is a union of cosets of a
subgroup of finite index in $G$, and hence a union of cosets of a normal
subgroup of finite index (the core of a finite index subgroup will be both
normal in $G$ and of finite index in $G$); it follows that recognisability of
a subset of $G$ is independent of the choice of generating set for $G$.
Rational subsets of $G$ can be alternatively characterised as those
sets that can be built out of finite subsets of $G$ using finite union
$A \cup B$, product $A \cdot B$, and semigroup generation $A^*$; it follows
from this that rationality of a subset of $G$ is also independent of the
choice of generating set.

By \cite[Theorems 3.1, 3.2]{BartholdiSilva}, 
a subgroup $H$ of $G$ is rational if and only if it is
finitely generated, and recognisable if and only if it has finite index;
the first of these results is attributed to Anisimov and Seifert.

We will be interested in the decidability of systems of equations and
inequations in which some of the variables are restricted to lying in
specified recognisable subsets of the group; this is the
\emph{Diophantine Problem with recognisable constraints}. This problem was considered for free groups and graph products in \cite{DiekertGutierrezHagenah}, \cite{DiekertLohrey08} and \cite{DiekertMuscholl05}.

Furthermore, if there exists an efficient algorithm that produces the solutions together with some concise and useful description of the solution set,
we say that the system is \emph{soluble} over $G$, or alternatively, that we can solve the
{\em Diophantine Search Problem} in $G$. We observe that this second definition is intrinsically
imprecise; in particular,
for countable groups $G$ with soluble word problem,
we can in principle enumerate the solutions by the naive method of testing
all $m$-tuples of elements of $G$.

The Diophantine Search Problem was solved in free groups by the work of Makanin and Razborov \cite{R85}, and descriptions of the solutions are possible via Makanin-Razborov diagrams, or as EDT0L formal languages \cite{CDE}.
Then \cite{DEtwisted} shows the solubility of equations with rational constraints in virtually
free groups, and \cite{CEhyp} the solubility of equations with q.i. embeddable rational constraints in
hyperbolic groups.
Further, we observe
that the set of solutions of a single equation or inequation over 
the free abelian group $\Z^n$ generated by $X = \{x_i: 1 \leq i \leq n\}$
can be expressed as a deterministic context-free language over the
alphabet $X \cup X^{-1}$.
So the solution set
of a system of equations and inequations over $\Z^n$ is the intersection of finitely many 
such languages. 

Our main result will rely on the following two statements about decidability of equations.

\begin{proposition}\label{prop:hypva}
(i) Let $G$ be a hyperbolic group.  Then systems of equations and inequations
with quasi-isometrically embeddable rational constraints are decidable in $G$.

(ii) Let $G$ be a virtually abelian group.  Then systems of equations
and inequations with recognisable constraints are decidable in $G$.
\end{proposition}
\begin{proof}
(i) This is proved in \cite[Theorem 1]{DahmaniGuirardel}.

(ii) We can adapt the proof of the same result, but without constraints, proved
in \cite[Lemma 5.4]{Dahmani}. That proof reduces the problem in $G$ to the same
problem in a free abelian subgroup $H$ of $G$. Given that any recognisable subset of $G$ can be written as a union of cosets of a finite index subgroup $M$,
we see that we can reduce our problem to the same problem in the free abelian
subgroup $H \cap M$. 

We note that systems of equations and inequations over the free abelian group $\Z^n$ are decidable and soluble using
standard techniques from integer linear algebra, such as the Hermite Normal
Form for matrices.
\end{proof}

\paragraph{Dihedral Artin groups.} A dihedral Artin group $\DA_m$, $m\geq 2$,
is defined by the presentation
\[ \DA_m = \langle a,b \mid aba\cdots = bab\cdots \rangle, \]
where the single (braid) relation relates the two distinct  alternating products
of length $m$ of the two generators $a,b$.
When $m$ is even, if we let $y_1:=a,y_2:=ab$, we see that
\[ \DA_m \cong \langle y_1,y_2 \mid y_1y_2^{m/2}=y_2^{m/2}y_1 \rangle, \]
and hence we can describe $\DA_m$ as a central extension of the infinite
cyclic group $\langle y_2^{m/2} \rangle$ by $\Z * C_{m/2}$. The latter group
has $\F_{m/2}$ as a subgroup of index $m/2$, and so in this
case $\DA_m$ is virtually the direct product  $\Z \times \F_{m/2}$.

When $m$ is odd, let $y_1:= ab\cdots a$, an alternating product of length
$m$, and $y_2:=ab$. Then
\[ \DA_m \cong \langle y_1,y_2 \mid y_1^2=y_2^m \rangle, \]
and so we can describe $\DA_m$ as a central extension of the infinite
cyclic group $\langle y_2^m \rangle$ by $C_2 * C_m$; the free product has
$\F_{m-1}$ as a subgroup of index $2m$.
In this case, $\DA_m$ is virtually $\Z \times F_{m-1}$.

\section{Main results}\label{sec:main}

\begin{theorem}
\label{thm:decide_abhyps}
Let $G$ be a finitely generated group that contains a direct product
$A \times H_1\times \cdots \times H_n$ as a finite index subgroup, where $A$
is virtually abelian and $H_1,\ldots, H_n$ are non-elementary hyperbolic.
Then systems of equations and inequations with recognisable constraints are
decidable over $G$.
\end{theorem}

We observe that since $\Z^n$ is a direct product of elementary hyperbolic
groups, $G$ can also be expressed as virtually a direct product of hyperbolic
groups.  We make no assumption in this result that $A$ is non-trivial
or that $n$ is non-zero.

Concerning solubility, it will be clear from the proof that descriptions
of the solution sets over the factors $H_i$ extend to descriptions in $G$.

\begin{corollary}
\label{cor:sol_DArtin_gp}
Systems of equations and inequations with recognisable constraints are
decidable in groups that are virtually dihedral Artin groups.
\end{corollary}

\begin{proofof}{corollary}
As explained in Section~\ref{sec:background}, every dihedral Artin group
is virtually a direct product of $F_m$ and $\Z$, for $m \geq 2$.
So the result follows from Theorem~\ref{thm:decide_abhyps}. 
\end{proofof}

We note that we cannot expect to improve Theorem \ref{thm:decide_abhyps} to
deal with rational constraints, as \cite[Prop.11]{DiekertLohrey04}, due to
Muscholl, shows. That result, concerning general right-angled Artin groups,
implies in particular that systems of equations with rational constraints over
direct products of non-abelian free groups are undecidable. The same holds for
q.i. embeddable rational constraints, as Muscholl's constraints can be made to consist
of geodesics.

We shall prove Theorem~\ref{thm:decide_abhyps} at the end of the paper.
It will follow from Proposition \ref{prop:hypva} and
Theorem~\ref{thm:decide_vdirprod_plus} below, once we have shown,
in Proposition~\ref{prop:abhyps_satisfy_plus}, that the group $G$ of
Theorem~\ref{thm:decide_abhyps} contains an appropriate normal subgroup $K$.

For a group $G$, we define $\FIN(G)$ to be the collection of groups each
of which is either (isomorphic to) a subgroup of finite index in $G$ or
contains (a group isomorphic to) $G$ as a subgroup of finite index.
We introduce this terminology since we observe that, for $G$ in various classes
of groups that will be considered in this article, the decidability and
solubility of systems of equations and inequations with various constraints
holds not just in $G$ but throughout $\FIN(G)$.

\begin{theorem}
\label{thm:decide_vdirprod_plus}
Let $G$ be a group and $K$ a finite index normal subgroup, where
$K=K_1 \times \cdots \times K_n$, and $\{K_1,\ldots,K_n\}$ is a union of
conjugacy classes of subgroups in $G$.
Suppose that systems of equations and inequations with recognisable constraints
are decidable in all groups in $\FIN(K_i)$, for each $i$. Then systems of
equations and inequations with recognisable constraints are decidable in $G$.
\end{theorem}
We need to establish some lemmas before proving this result.
\begin{lemma}
\label{lem:dirprod_ineq}
Let $G=G_1\times \cdots \times G_n$ be a direct product of groups $G_i$ over
which systems of equations and inequations with recognisable constraints are
decidable. Then the same is true in $G$.
\end{lemma}
\begin{proof}
Suppose that the system consists of a set $E$ of equations and a set $I$ of 
inequations. An equation from $E$ has a solution in $G$ if and only if,
in each of the direct factors $G_i$, the projection onto $G_i$ has a solution,
while an inequation from $I$ has a solution in $G$ if and only, 
in at least one of the direct factors $G_i$, the projection onto $G_i$ 
has a solution.
Hence the system has a solution in $G$ if and only if we can write 
$I$ as a (not necessarily disjoint) union $I_1 \cup I_2 \cdots \cup I_n$,
where for each $i$ the projection of the system $E \cup I_i$ onto $G_i$
has a solution in $G_i$.  So decidability
in $G$ is inherited from decidability in the direct factors $G_i$.

Now suppose that some of the variables are restricted to lie in some
specified recognisable subsets of $G$.  We recall that each such subset
is a finite union of cosets of a finite index normal subgroup $M_j$ of $G$
and by letting $H_i := \cap_j M_j \cap G_i$,
we have $H_i \lhd G_i$, $|G_i:H_i|<\infty$,
and 
$H:=H_1 \times \cdots \times H_n$ is contained in all of the subgroups
that arise in the constraints.

We first find all solutions of the projection of the system onto the finite
quotient $G/H$ of $G$.  Then, for each such solution, the set of solutions of
the original system that project onto it can be defined in terms of the
solutions of a system of equations and inequations over $G$ that are
constrained to lie in $H$.  As in the first paragraph of this proof, we can
reduce the decidability of such a systems to the decidability of a finite collection of
systems of equations and inequations over the component groups $G_i$ for which
the solutions are constrained to lie in $H_i$, and we can decide each of those
by hypothesis.
\end{proof}

The following lemma shows how one can embed an extension of a direct product
into a direct product of wreath products, and has the flavour of a 
Kaloujnine-Krasner result \cite{KaloujnineKrasner} about embeddings of 
group extensions into wreath products; however, we do not specifically need
that result here.

\begin{lemma}\label{lem:embed_wr_prod}
Let $K= K_1\times \cdots \times K_n$ be a normal subgroup of finite index in a
group $G$, and suppose that the set of subgroups $\{K_1,\ldots,K_n\}$ is a 
union of conjugacy classes of subgroups in $G$.
\begin{mylist}
\item[(i)] If the subgroups $K_i$ form a single conjugacy class,
then $G$ is isomorphic to a subgroup of finite index in $J \wr P$,
where $J \cong N_G(K_1)/(K_2\times \cdots \times K_n)$ contains a subgroup
of finite index isomorphic to $K_1$, and $P \le \Sym_n$ is 
the image of the permutation action on
$\{K_1,...,K_n\}$ induced by conjugation in G.

\item[(ii)] Suppose that $K_1,\ldots K_k$ are representatives of the conjugacy
classes of $K_1\ldots,K_n$ within $G$. Then $G$ embeds as a subgroup of finite
index in a direct product $W_1\times W_2\times\cdots \times W_k$ of
wreath products $W_j= J_j \wr P_j$, where $J_j$ is a group containing $K_j$ as
a subgroup of finite index and $P_j$ is a finite permutation group.
\end{mylist}
\end{lemma}
\begin{proof}
Part (i) is a rewording of \cite[Theorem 4.1 (1)]{GrossKovacs}.

For (ii), for $1 \le i \le k$, let $N_i$ be the product of those $K_j$ that are
not conjugate to $K_i$ in $G$. Then $N_i \unlhd G$, and we define $Q_i:=G/N_i$ 
and let $\mu_i:G \to Q_i$ be the natural map. Then the images under $\mu_i$
of those $K_j$ that are conjugate in $G$ to $K_i$ form a single conjugacy
class of subgroups of $Q_i$, and their product has finite index in $Q_i$.
So by (i) $Q_i$ embeds via a map $\eta_i$ as a subgroup of finite index in a
group $W_i = J_i \wr P_i$, where $J_i$ contains $\mu_i(K_i) \cong K_i$ as a
subgroup of finite index, and $P_i$ is a finite permutation group.

Now the map $\mu:G \to W_1 \times \cdots \times W_k$ defined by
$\mu(g) = (\eta_1\mu_1(g),\ldots,\eta_k\mu_k(g))$ is an embedding of $G$ into
$W_1\times  \cdots \times W_k$. Since, for $i=1,\ldots,k$, $\mu(K_i)$ is a
subgroup of the direct factor $W_i$ isomorphic to $K_i$, and $\mu(K_i^G)$ has
finite index in $W_i$, we see that $\mu(K)$ (and hence also $\mu(G)$) has
finite index in $W_1\times \cdots \times W_k$.
\end{proof}

\begin{lemma}\label{lem:wr_prod} Suppose that systems of equations and
inequations with recognisable constraints are decidable over the group $J$.
Then they are also decidable over the permutation wreath product
$W= J \wr P$ of $J$ with a finite subgroup $P$ of $\Sym_n$.
\end{lemma}

\begin{proof}
Decomposing $W$ as the split extension of $n$ copies of $J$ by
$P \subset \Sym_n$, we represent each of its elements by an
$(n+1)$-tuple $(j_1,\ldots,j_n,\pi)$,
with $j_i \in J$, $\pi \in P$, with multiplication defined by
\[ (j_1,\ldots,j_n,\pi)(k_1,\ldots,k_n,\rho) =
      (j_1k_{\pi^{-1}(1)},\ldots,j_nk_{\pi^{-1}(n)},\pi\rho).\] 

For a given system of equations and inequations over $W$, we project this
system onto the finite group $P$ and find all of the finitely many
solutions in $P$. For each such solution in $P$, we can use the
displayed equation to reduce the problem of deciding whether there are
solutions in $W$ that project onto that particular solution in $P$
to one of deciding a system of equations and inequations over the direct product
$J^n$ of $n$ copies of $J$.  We can do that by Lemma \ref{lem:dirprod_ineq}.

If the system over $W$ has rational constraints, then this technique reduces
the problem to one of deciding systems of equations with rational constraints over
$J^n$, which we can again do by Lemma \ref{lem:dirprod_ineq}.
\end{proof}

\begin{proofof}{Theorem \ref{thm:decide_vdirprod_plus}}
Suppose that the groups $K_1,\ldots,K_n$ fall into $k$ conjugacy classes
under the conjugation action of $G$, of which $K_1,\ldots,K_k$ are
representatives, and of the $n$ original subgroups, $n_j$ of them are
conjugate to $K_j$, for each $1\leq j \leq k$.  Then, by
Lemma~\ref{lem:embed_wr_prod}, $G$ embeds as a subgroup of finite index
in a direct product $W_1\times W_2\times \cdots \times W_k$ of permutation
wreath products $W_j= J_j \wr P_j$, where $J_j$ is a group containing $K_j$ as
a subgroup of finite index and $P_j$ is a subgroup of $\Sym_{n_j}$.
Our hypotheses ensure that, for each $j=1,\ldots,k$,
equations with recognisable constraints are decidable in $J_j$, and
Lemma~\ref{lem:wr_prod} ensures that the same is true in $W_j$.

By Lemma~\ref{lem:dirprod_ineq} systems of equations with recognisable
constraints are decidable if that holds for the direct factors in
the direct product $W_1\times W_2\cdots \times W_k$. 
Then since finite index subgroups are recognisable, it follows that the same
is true within the finite index subgroup $G$ of $W_1 \times\cdots \times W_k$.
\end{proofof}

\section{Direct products of finite index}
\label{sec:dpfi}
Recall that virtually cyclic groups (including finite groups) are hyperbolic,
and are known as {\em elementary hyperbolic}. The following lemma lists
the known properties of hyperbolic groups that we shall need.

\begin{lemma}\label{lem:hypfacts}
Let $H$ be a hyperbolic group.  Then the centralizer of any element of $H$
is hyperbolic, and is elementary if the element is non-torsion.
Any subgroup of $H$ consisting of torsion elements is finite,
and there is a bound on the order of the finite subgroups of $H$.
Furthermore, if $H$ is non-elementary, then  $Z(H)$ is finite and $H/Z(H)$
is non-elementary hyperbolic.
\end{lemma}
\begin{proof} 
The first two assertions are proved in
\cite[Proposition 4.3 and Proposition 5.1]{GerstenShort}.
The finiteness of torsion subgroups is proved in
\cite[Corollaire 36, Chapitre 8]{GhysHarpe}
and the bound on the order of finite subgroups is proved in
\cite[Theorem III $\Gamma$.3.2]{BridsonHaefliger}.
It follows that $Z(H)$ is a torsion subgroup when $H$ is non-elementary, so
$Z(H)$ is finite, and the proof that $H/Z(H)$ is hyperbolic and non-elementary
is straightforward.
\end{proof}

\begin{lemma}\label{lem:conjhyp}
Let $A$ be virtually abelian, and $H_1,\ldots,H_n$ non-elementary hyperbolic
groups.
Let $H = A \times H_1 \times \cdots \times H_n$, let $L \le H$ with $|H:L|$
finite, and let $g \in L$. Suppose that the projection of $g$ onto $H_i$
is a non-torsion element of $H_i$ for exactly $k$ values of
$i \in \{1,2,\ldots,n\}$. Then the centraliser of $g$ in $L$ is a subgroup
of finite index in a group $B \times K_1 \times \cdots \times K_j$,
where $B$ is virtually abelian, each $K_i$ is non-elementary hyperbolic,
and $j \le n-k$.
\end{lemma}
\begin{proof} This follows from the previous lemma.
\end{proof}

\begin{lemma}\label{lem:numhypfacs}
Let $A$ and $B$ be virtually abelian, and
$H_1,\ldots,H_m,K_1,\ldots,K_n$ non-elementary hyperbolic
groups.
Let $H = A \times H_1 \times \cdots \times H_m$ and
$K = B \times K_1 \times \cdots \times K_n$.
Suppose that a group $L$ is isomorphic to finite index subgroups of
both $H$ and $K$. Then $m=n$.
\end{lemma}
\begin{proof}
Suppose that $m \le n$ and use induction on $m$. If $m=0$, then $H$ is
virtually abelian, and hence $K$ must be virtually abelian, and so $n=0$.
So suppose that $m>0$.

It is convenient to identify $L$ with the subgroups of $H$ and $K$ with
which it is isomorphic.
By replacing $H$, $K$ and $L$ by finite index subgroups, we can assume 
that $A$ and $B$ are both abelian, and that $L$ projects onto
all of the direct factors $H_i$ and $K_i$.
Then $Z(L) = L \cap Z(H) = L \cap Z(K)$ and, since $L \cap A \le Z(L)$ 
and $L \cap B \le Z(L)$, it follows from Lemma~\ref{lem:hypfacts} that
$\overline{L} := L/Z(L)$ can be identified with finite index subgroups of
$\overline{H_1} \times \cdots \times \overline{H}_m$ and of
$\overline{K_1} \times \cdots \times \overline{K}_n$, where 
$\overline{H_i} := H_i/Z(H_i)$ and $\overline{K_i} := K_i/Z(K_i)$ are
non-elementary hyperbolic groups.

Since $\overline{L} \cap \overline{K}_1$ has finite index in $\overline{K}_1$
it contains a non-torsion element $g$, and $C_{\overline{L}}(g)$ has finite
index in a direct product of a virtually abelian group and $n-1$ non-elementary
 hyperbolic groups. But, by considering $g$ as an element of
$\overline{H_1} \times \cdots \times \overline{H}_m$, we see from
Lemma~\ref{lem:conjhyp} that $C_{\overline{L}}(g)$ has finite index in the
direct product of a virtually abelian group and $m-t$ non-elementary hyperbolic
groups for some $t \ge 1$. So, by the inductive hypothesis, we have $m-t=n-1$,
and since $m \le n$, we must have $t=1$ and $m=n$.
\end{proof}

\begin{proposition}
\label{prop:abhyps_satisfy_plus}
Let $A$ be virtually abelian, and $H_1,\ldots,H_n$ non-elementary hyperbolic
groups.  Suppose that a group $G$ has a subgroup $H$ of finite index with
$H \cong A \times H_1 \times \cdots H_n$. Then there is a normal subgroup
$K$ of finite index in $G$ with $K \le H$, such that
$K \cong B \times K_1 \times \cdots K_n$, where $B \le A$ of finite index in
$A$, each $K_i$ is isomorphic to a finite index subgroup of $H_i$, and the set
$\{B,K_1,\ldots,K_n\}$ is a union of orbits under conjugation by $G$.
\end{proposition}
\begin{proof}
By replacing $H$ by a subgroup of finite index, we may assume that $A$
is free abelian. Let $N$ be the core of $H$ in $G$; then $N \unlhd G$ has
finite index and $N \leq H$.

Let $C :=  N \cap A$; then $|A:C| \leq |G:N|< \infty$. Also, $C \le Z(N)$, and
$C$ is torsion free.  The projection
of $Z(N)$ onto each of the subgroups $H_i$ is central in a subgroup of
finite index in $H_i$ and hence finite, so $|Z(N):C|<\infty$.
Let $k:=|Z(N):C|$ and define $B:= Z(N)^k$. Then $B \le C$ and $B$ is
characteristic and of finite index in $Z(N)$. So $B$ has finite index in $C$,
and hence in $A$, and since $Z(N)$ is normal in $G$, so is $B$.  

The rest of this proof is devoted to the construction of the subgroups
$K_1,\ldots,K_n$ of $G$. We find these as subgroups of finite index in
subgroups $L_1,\ldots,L_n$ of $G$, which we identify by considering a quotient
$G/T$ of $G$, and considering its action by conjugation on its free abelian
normal subgroup $N/T$.  

We note that $[N,N] \le H_1 \times \cdots H_n$, so $B \cap [N,N] = 1$.
Now choose $T$ with $[N,N] \le T \le N$ so that
$T/[N,N]$ is the torsion subgroup of the abelian group $N/[N,N]$; then
$N/T$ is free abelian. 
Since $T/[N,N]$ is characteristic in $N/[N,N]$, and $N/[N,N]$ is normal in
$G/N$, certainly $T \unlhd G$.  Since $B$ is torsion-free with $B \cap
[N,N]=1$, while $T/[N,N]$ has torsion, we have $T \cap B = 1$, and so $BT/T
\cong B$.  Furthermore, since $B$ is normal in $G$, the image
$BT/T \cong B$ of $B$ in $G/T$ is normal in $G/T$. 

Let $g \mapsto \bar{g}$ denote the natural map from $G$ to $G/T$.
Then the conjugation action of $\overline{G}$ on the free abelian group
$\overline{N}$
makes $\overline{N}$ into a torsion-free $\Z\overline{G}$-module in which
$\overline{BT}$ is a submodule. So by Lemma \ref{lem:bymaschketor} below, 
there is a subgroup $\overline{L}$ of $\overline{N}$ with
$\overline{L} \unlhd \overline{G}$,
$|\overline{N}: \overline{L}\overline{BT}| <\infty$,
and $\overline{L} \cap \overline{BT}= \{ 1 \}$.
Then, where $L\unlhd G$ is the preimage of $\overline{L}$,
we deduce that $|N:BL|<\infty$ (and so also $|H:AL|<\infty$)  
and $L \cap BT=T$. Since $T \cap B = 1$,
we have $L \cap B = 1$, and also $L \cap A=1$ (since if $g \in L \cap A$, we
have $g^k \in L \cap B = \{ 1 \}$, so $g=1$, since $A$ is torsion-free).
Now the natural map from $H$ to $H/A$, whose image is isomorphic to
$H_1 \times \cdots H_n$,
maps $AL$ to a group of finite index in $H/A$, which is isomorphic to $L$;
the image lifts to a subgroup $M$ of finite index in
$H_1 \times \cdots \times H_n$, to which we associate
an isomorphism $\phi:L \to M$.
For each $i$, we define $L_i := \phi^{-1}(M \cap H_i)$.
Then $L_i$ is isomorphic to a subgroup of finite index in $H_i$,
$L_1 \times \cdots \times L_n$ has finite index in $L$,
and $B \times L_1 \times \cdots \times L_n$ has finite index in $G$.

Let $h$ be a non-torsion element of $L_i$ for some $i$.
Then $C_L(h) \cong C_M(\phi(h))$
has finite index in the direct product of a virtually cyclic group and
$n-1$ non-elementary hyperbolic groups. Now, for any $g \in G$, the
same applies to $h^g = g^{-1}hg$ and so,
by Lemmas~\ref{lem:conjhyp} and \ref{lem:numhypfacs} 
applied to $\phi(h^g)$, we see that the projection of $\phi(h^g)$ onto
$H_j$ is a non-torsion element for exactly one value of $j$.

Furthermore, if $h'$ is another non-torsion element of the same $L_i$, then
$C_L(\langle h,h' \rangle)$ has finite index in the direct product of a
(possibly finite) virtually cyclic group and $n-1$ non-elementary
hyperbolic groups, and again the same applies to
$C_L(\langle h^g,h'^g \rangle)$. 
It follows that the unique subgroup $H_j$ onto which the projection of
$\phi(h^g)$ is a non-torsion element is the same as that onto which
the projection of $\phi(h'^g)$ is a non-torsion element;
hence we may denote that subgroup by $H_{i^g}$.
We notice too that we can find an integer $r$ for which the projections
of $\phi((h^g)^r)=\phi((h^r)^g)$ onto all subgroups $H_j$ apart from $H_{i^g}$
are trivial. So we have $h^r \in L_i$, and $(h^r)^g \in L_{i^g}$.
Of course, for all $g'$, we also have 
$\phi(((h^r)^{gg'})) = \phi(((h^r)^g)^{g'})$.
It follows that the unique subgroup $H_r$ onto which the projection
of  $\phi(h^r)^{gg'}$ is a non-torsion element is identified both as
$H_{i^{gg'}}$ and as $H_{(i^g)^{g'}}$, and so we see that, for all $g,g'$
we have $i^{gg'}=(i^g)^{g'}$.

Now, for any element $h \in L_i$, whether or not it is a torsion element,
and for any $j \ne i^g$, the projection of $\phi(h^g)$ onto $H_j$ is a torsion
element. So the projection of $\phi(L_i^g)$ onto $H_j$ is a torsion group
and hence, by Lemma \ref{lem:hypfacts} is finite.
Now let $P_i$ denote the intersection of the kernels of all homomorphisms 
$\theta: L_i \rightarrow H_i$ for which $|\theta(L_i)| < \infty$.
If $h \in P_i$, then $\phi(h^g) \in H_{i^g}$ and hence $h^g \in L_{i^g}$.
Also, since by Lemma \ref{lem:hypfacts} there is a bound on the orders of finite
subgroups of $H_i$, $|L_i:P_i|$ is finite.

Finally, let $K_i = \{ h \in L_i \mid h^g \in L_{i^g}\,\forall g \in G\}$.
Then it is straightforward to check that $K_i$ is a subgroup of $G$ and,
since $P_i \le K_i$,  we see that $|L_i:K_i|$ is finite; hence $K_i$ is
isomorphic to a subgroup of finite index in $H_i$.  It follows from the
statement $i^{gg'}=(i^g)^{g'}$ that $K_i^g \le K_{i^g}$ for all $i$ and $g$
and then, since $(K_{i^g})^{g^{-1}} \le K_i$, we must have $K_i^g = K_{i^g}$.
So we have proved that $\{K_1,\ldots,K_n\}$ is a union of orbits under the
conjugation action of $G$.

This completes the proof.
\end{proof}

\begin{lemma}\label{lem:bymaschketor}
Let $G$ be a finite group, let $V$ be a finite dimensional
torsion-free $\Z G$-module, and $W$ a submodule. Then there exists
a $\Z G$-submodule $U$ of $V$
with $U \cap W = \{0\}$ such that $V/(U \oplus W)$ is finite.
\end{lemma}
\begin{proof}
Let $\widehat{V}=V \otimes \Q$ and $\widehat{W}=W \otimes \Q$ be the
corresponding $\Q G$-modules. By Maschke's therem, there exists a
$\Q G$-submodule $\widehat{U}$ of $\widehat{V}$ with
$\widehat{V}=\widehat{U} \oplus \widehat{W}$.  Let $e_1,\ldots,e_n$ be a
$\Z$-basis of $V$, which we may consider also as a $\Q$-basis of
$\widehat{V}$. We can choose a basis $u_1,\ldots,u_k$ of $\widehat{U}$ such that
the matrices representing the action of $G$ have integer entries.
Define $\lambda_{ij} \in \Q$ by $u_i=\sum_{j=1}^n \lambda_{ij} e_j$.
Let $m$ be a common multiple of the denominators of all the $\lambda_{ij}$,
and define $U \subseteq V$ to be the $\Z$-module generated by the elements
$\mu_i,\ldots,\mu_k$ of $\widehat{U}$. Then $U \oplus W$ has rank $n$, and so
must have finite index in $V$.
\end{proof}

\begin{proofof}{Theorem~\ref{thm:decide_abhyps}}
Let $G$ be as in the hypothesis of the theorem. Then, by
Proposition \ref{prop:abhyps_satisfy_plus}, $G$ has a normal subgroup
$K$ of finite index, such that $K \cong B \times K_1 \times \cdots K_n$, where
$B \le A$ with $|A:B|$ finite,
each $K_i$ is isomorphic to a finite index subgroup of $H_i$, and the set
$\{B,K_1,\ldots,K_n\}$ is a union of orbits under conjugation by $G$.
Now all groups in $\FIN(B)$ are virtually abelian, and all groups
in $\FIN(K_i)$ are hyperbolic for each $i$, so systems of equations and
inequations with recognisable constraints are decidable in all
groups in either $\FIN(B)$ or $\FIN(K_i)$ (any $i$).
The result now follows
by Theorem \ref{thm:decide_vdirprod_plus}.
\end{proofof}

\textsc{Laura Ciobanu,
Mathematical and Computer Sciences,
Colin McLaurin Building, 
Heriot-Watt University,      
Edinburgh EH14 4AS, UK}

\emph{E-mail address}{:\;\;}\texttt{l.ciobanu@hw.ac.uk}
\bigskip

\textsc{Derek Holt,
Mathematics Institute,
Zeeman Building,
University of Warwick,
Coventry CV4 7AL, UK
}

\emph{E-mail address}{:\;\;}\texttt{D.F.Holt@warwick.ac.uk}
\bigskip

\textsc{Sarah Rees,
School of Mathematics and Statistics,
University of Newcastle,
Newcastle NE1 7RU, UK
}

\emph{E-mail address}{:\;\;}\texttt{Sarah.Rees@ncl.ac.uk}
\end{document}